\documentclass[11pt]{amsart}

\usepackage{amsfonts, amstext, amsmath, amsthm, amscd, amssymb}
\usepackage{enumerate}

\let\oldmarginpar\marginpar
\renewcommand\marginpar[1]{\oldmarginpar[\raggedleft\footnotesize #1]%
{\raggedright\footnotesize #1}}

 \usepackage[
 pdfborderstyle={},
 pdfborder={0 0 0},
 pagebackref,
 pdftex]{hyperref}

 \renewcommand*{\backref}[1]{}
 \renewcommand*{\backrefalt}[4]{
   \ifcase #1 %
    [No citations.]%
   \or
    [#2]%
   \else
    [#2]%
   \fi
 }

\newcommand{\RR}{{\mathbb{R}}}
\newcommand{\ZZ}{{\mathbb{Z}}}

\renewcommand{\setminus}{{\smallsetminus}}
\newcommand{\bdy}{\partial}
\newcommand{\area}{{\mathrm{area}}}
\newcommand{\abs}[1]{{\left\vert #1 \right\vert}} 
\def\co{\colon\thinspace}
\newcommand{\lmin}{{\ell_{\min}}}
\newcommand{\kmax}{{\kappa_{\max}}}
\newcommand{\tlim}{{t_{\lim}}}

\theoremstyle{plain}
\newtheorem{theorem}{Theorem}[section]

\newtheorem{lemma}[theorem]{Lemma}

\newtheorem{claim}[theorem]{Claim}

\newtheorem*{namedtheorem}{\theoremname}
\newcommand{\theoremname}{testing}

\theoremstyle{definition}
\newtheorem{define}[theorem]{Definition}

\numberwithin{equation}{section}

\begin{document}

\title{Explicit Dehn filling and Heegaard splittings}
\author{David Futer}

\address[]{Department of Mathematics, Temple University, Philadelphia,
	PA 19122, USA}

\email[]{dfuter@temple.edu}
 \thanks{Futer is  is supported in part by NSF grant DMS--1007221. }

\author{Jessica S. Purcell}

\address[]{ Department of Mathematics, Brigham Young University,
Provo, UT 84602, USA}

\email[]{jpurcell@math.byu.edu }
\thanks{Purcell is supported in part by NSF grant 
 DMS--1007437 and the Alfred P. Sloan Foundation.}

\thanks{ \today}

\begin{abstract}
We prove an explicit, quantitative criterion that ensures the Heegaard
surfaces in Dehn fillings behave ``as expected.'' Given a cusped
hyperbolic $3$--manifold $X$, and a Dehn filling whose meridian and
longitude curves are longer than $2\pi(2g-1)$, we show that every
genus $g$ Heegaard splitting of the filled manifold
is isotopic to
a splitting of the original manifold $X$. The analogous statement holds
for fillings of multiple boundary tori.  This gives an effective
version of a theorem of Moriah--Rubinstein and Rieck--Sedgwick.
\end{abstract}

\maketitle

\section{Introduction}

In 1997, Moriah and Rubinstein investigated the relationship between
Heegaard splittings of a cusped hyperbolic $3$--manifold and the
Heegaard splittings of its Dehn fillings \cite{moriah-rubinstein}. 
They showed that if one imposes a bound on the genus of the surfaces 
and excludes finitely many Dehn filling slopes, 
then every irreducible Heegaard
surface in the filled manifold is isotopic to one of a finite
collection of surfaces in the original manifold.  In 2001, Rieck and
Sedgwick used topological ideas to show that any Dehn filling slope
that results in a smaller genus Heegaard surface must lie on one of a
finite number of so-called ``bad'' slopes and ``destabilization
lines'' in Dehn surgery space \cite{rieck-sedgwick}.  Rieck showed
that the number of bad slopes is bounded by a quadratic function of
the genus \cite{rieck}, while Rieck and Sedgwick showed that this
number is finite in general, independent of genus
\cite{rieck-sedgwick2}.

Due to this previous work, we know that if we exclude a finite number
of Dehn filling slopes, and a finite number of destabilization lines
in the Dehn surgery space, any bounded--genus Heegaard surface in a
Dehn filling will be a Heegaard surface in the original manifold.
However, there has not been an effective characterization of
\emph{which} slopes and destabilization lines must be excluded.  As a
consequence, it has been difficult to use these results to prove
explicit bounds, such as those needed in a recent paper of the authors
with Cooper \cite{cooper-futer-purcell}.  

In this paper, we make these constraints explicit.  In particular, we
show the following.

\begin{theorem}\label{thm:main}
Let $X$ be a cusped, orientable hyperbolic 3--manifold. Choose
disjoint horospherical neighborhoods $C_1, \ldots, C_k$ about some
subset of the cusps, and let $s_i$ be a Dehn filling slope on each
torus $\bdy C_i$. Let $\Sigma$ be a Heegaard surface of genus $g \geq
1$ for the Dehn filled manifold $M = X(s_1, \ldots, s_k)$.  Then
we have the following.
\begin{enumerate}
\item\label{i:core-in-surf} If the length $\ell(s_i)$ satisfies
  $\ell(s_i) > 2\pi(2g-1)$ for every $i$, then each core curve
  $\gamma_i$ for the filling solid torus is isotopic into $\Sigma$
  (although these cores may not be simultaneously isotopic into
  $\Sigma$).
\item\label{i:surf-from-unfilled} If, in addition, the shortest
  longitude $\lambda_i$ for each $s_i$ satisfies $\ell(\lambda_i) >
  6(2g-3)$, the surface $\Sigma$ can be isotoped into $M \setminus
  (\gamma_1 \cup \ldots \cup \gamma_k) \cong X$, and forms a Heegaard
  surface for $X$.
\end{enumerate}
\end{theorem}

A \emph{longitude for $s_i$} in item \eqref{i:surf-from-unfilled} is
defined to be a slope $\lambda_i$ on $\bdy C_i$ that intersects $s_i$
once.  The \emph{shortest longitude} is a longitude whose length is
smallest among all longitudes for $s_i$.  Here the lengths $s_i$ and
$\lambda_i$ are the lengths of geodesic representatives on the
horospherical torus $\partial C_i$,
in the metric induced by the hyperbolic metric on $X$.

The slopes on $\bdy C_i$ that fail condition \eqref{i:core-in-surf} of
the theorem, i.e.\ those slopes whose core of the Dehn filling solid
torus is not isotopic into $\Sigma$, are exactly the ``bad'' slopes
studied by Rieck and Sedgwick \cite{rieck, rieck-sedgwick2,
  rieck-sedgwick}. Thus part \eqref{i:core-in-surf} of Theorem
\ref{thm:main} gives an explicit finite list of candidates for bad
slopes. Similarly, for each slope $\lambda_i$ of length less than
$6(2g-3)$, the primitive integer coordinates of the meridians that
intersect $\lambda_i$ once lie on a single line in $\RR^2$; this line
is exactly the ``destabilization line'' corresponding to
$\lambda_i$. Thus part \eqref{i:surf-from-unfilled} of Theorem
\ref{thm:main} gives an explicit finite list of candidates for
destabilization lines.

We note that because $g \geq 1$, the hypotheses of the theorem always
require that the Dehn filling slopes satisfy $\ell(s_i) > 2\pi$. As a
result, the $2\pi$--Theorem of Gromov and Thurston
\cite{bleiler-hodgson} implies that the manifold $M=X(s_1, \ldots,
s_k)$ has a negatively curved metric. (Thus, in fact, the case $g=1$
is vacuous.)

Our main tool in proving Theorem \ref{thm:main} is geometry: we obtain
our conclusions from area considerations in the negatively curved
metric on $M$.  This geometric viewpoint follows the lead of Moriah
and Rubinstein's paper \cite{moriah-rubinstein}.  We also follow their
lead in using a theorem of Pitts and Rubinstein
\cite{pitts-rubinstein:min2, rubinstein:survey} that relates Heegaard
surfaces to minimal surfaces in $M$ (see Lemma \ref{lemma:sweepout}
below).
The Pitts--Rubinstein result has been used in a number of applications
(e.g.\ \cite{breslin, lackenby:Heegaard, maher:virtual-fibers,
  moriah-rubinstein}),
but unfortunately a complete proof of this result does not appear in
the literature.
De Lellis and Pellandini have proved an important step, namely that
the ``minimax method'' produces a minimal surface of the appropriate
genus \cite{delellis-pellandini}.  In addition, a survey paper by
Souto \cite{souto:heegaard-survey} contains a discussion of the status
of the proof, including what remains to be done after the work of De
Lellis and Pellandini. Souto claims in \cite{souto:heegaard-survey} to have worked out the remainder
of the proof, although his argument has not yet appeared. In this
work, we shall assume the Pitts--Rubinstein result, in anticipation of
a full proof.

In addition to methods used by Moriah and Rubinstein, our argument
takes advantage of several other tools, some of which were unavailable
in 1997.

First, we will use an effective version of the $2\pi$--Theorem, due to
the authors and Kalfagianni \cite[Theorem 2.1]{fkp:jdg}, to get
explicit estimates on curvature and area in $M$. This result,
described in Section \ref{sec:Dehn}, will give us much more effective
control over surfaces.

Second, we will use the notion of \emph{generalized Heegaard
  splittings}, developed by Scharlemann and Thompson
\cite{scharlemann-thompson}, to reduce the crux of the argument to the
case where $\Sigma$ is strongly irreducible. We review the relevant
ideas in Section \ref{sec:sweepouts}.

Third, our proof relies on the argument in a recent paper by Breslin
\cite{breslin}.  In fact, we obtain a generalization of his theorem,
which is likely to be of independent interest.  To simplify the
statement of our generalization, we use the following definition.

\begin{define}\label{def:submersible}
Let $V$ be a solid torus, with a prescribed Riemannian metric. We say
the metric on $V$ is \emph{submersible} if, after lifting the metric
to the universal cover $\widetilde{V} \cong D \times \RR$, there is a
Riemannian submersion from $\widetilde{V}$ to its cross-sectional disk
$D$. Recall that a smooth map $f: \widetilde{V} \to D$ is called a
Riemannian submersion if its differential $df\co T_p \widetilde{V} \to
T_{f(p)} D$ is an orthogonal projection at each point.
\end{define}

We note that if $M$ has a hyperbolic metric and $V \subset M$ is a
fixed-radius tube about a closed geodesic, then the metric on $V$ is
submersible. In addition, we will see in Section \ref{sec:Dehn} that
the negatively curved solid tori constructed using the $2\pi$--Theorem
are also submersible. Thus the hypothesis of submersibility is
relatively mild.

\begin{theorem}\label{thm:intro-breslin}
  Let $M$ be an orientable Riemannian $3$--manifold whose non-compact
  ends (if any) are isometric to horospherical cusp neighborhoods, and
  whose boundary (if any) consists of minimal surfaces.  Suppose the
  sectional curvatures of $M$ are bounded above by $\kmax < 0$. Let
  $\Sigma$ be a Heegaard surface for $M$. Let $V$ be a solid torus in
  $M$, such that the metric on $V$ is submersible and its
  cross--sectional disk $D$ satisfies
  \begin{equation}\label{eq:needed-area}
    \area(D) \: > \: 2\pi \chi(\Sigma) / \kmax.
  \end{equation}
  Then the core curve $\gamma$ of $V$ is isotopic into $\Sigma$. Here
  $\chi(\Sigma)$ denotes Euler characteristic.
\end{theorem}

Theorem \ref{thm:intro-breslin} generalizes Breslin's theorem in
several ways. First, it does not require $\Sigma$ to be strongly
irreducible.  Second, it allows $M$ to have cusps and/or
boundary. Third, it allows the metric on $M$ to have variable
curvature. Finally, the explicit hypothesis of equation
\eqref{eq:needed-area} may be easier to check in applications than
Breslin's hypothesis on the length of $\gamma$.  These improvements
are obtained by modifications of Breslin's original argument; we
describe them in Section \ref{sec:breslin}.

With all of this work in hand, part \eqref{i:core-in-surf} of Theorem
\ref{thm:main} will follow immediately by applying Theorem
\ref{thm:intro-breslin} to the negatively curved metric described in
Theorem \ref{thm:negcurv-metric}. Part \eqref{i:surf-from-unfilled} of
Theorem \ref{thm:main} will follow by another geometric argument,
which is given in Section \ref{sec:assembly}.

\subsection*{Acknowledgments}
We thank Marc Lackenby, Yo'av Rieck, Saul Schleimer, and Juan Souto for a number of
helpful conversations.  We also thank the referee for numerous  suggestions that improved our exposition.

\section{A negatively curved metric on the filled manifold}\label{sec:Dehn}

In this section, we begin with a cusped hyperbolic manifold $X$, and
recall an explicit construction of a negatively curved metric on a
Dehn filling $M=X(s_1, \ldots, s_k)$.  Following results of the the
authors and Kalfagianni in \cite[Section 2]{fkp:jdg}, we will obtain
explicit estimates on curvature and areas in the negatively curved
metric on $M$.  These estimates are designed to plug
into equation
\eqref{eq:needed-area}, which will
give us control over Heegaard surfaces in
$M$.

\begin{theorem}\label{thm:negcurv-metric}
  Let $X$ be a complete, finite--volume hyperbolic manifold with
  cusps.  Suppose $C_1, \dots, C_k$ are disjoint horoball
  neighborhoods of some (possibly all) of the cusps.  Let $s_1, \dots,
  s_k$ be slopes on $\partial C_1, \dots, \partial C_k$, each with
  length greater than $2\pi$.  Denote the minimal slope length by
  $\lmin$.
  
  Then, for every $\zeta \in (0,1)$, the Dehn filled manifold $M =
  X(s_1, \ldots, s_k)$ admits a Riemannian metric, in which the cusp
  $C_i$ is replaced by a negatively curved solid torus $V_i$. This
  metric has the following properties:
  \begin{enumerate}
  \item\label{item:hyp} The metric on $M \setminus \bigcup_i V_i$
    agrees with the hyperbolic metric on $X \setminus \bigcup_i C_i$.
  \item\label{item:sec} The sectional curvatures of $M$ are bounded
    above by $\zeta\left( \left(\frac{2\pi}{\lmin}\right)^2 - 1\right)
    < 0$.
  \item\label{item:sym} The metric on each $V_i$ is submersible, as in
    Definition \ref{def:submersible}.
  \item\label{item:merid} The cross-sectional disk of $V_i$ has area
    at least $\zeta \left(\dfrac{\lmin^2}{\lmin + 2\pi}\right)$.
  \end{enumerate}
\end{theorem}

\begin{proof}
To construct a Riemannian metric on $M$ satisfying \eqref{item:hyp},
it suffices to construct a negatively curved metric on each solid
torus $V_i$, such that in a collar neighborhood of $\bdy V_i$, it
agrees with the hyperbolic metric in a collar neighborhood of $\bdy
C_i$. This is precisely what is done in \cite[Theorem
  2.1]{fkp:jdg}. Thus the cusp neighborhoods $C_i$ can be replaced by
solid tori $V_i$, obtaining \eqref{item:hyp}. In addition, we showed
in \cite[Lemma 2.3 and Theorem 2.5]{fkp:jdg} that the curvatures of
the resulting metric on $M$ are bounded as claimed in
\eqref{item:sec}.

For the rest of the argument, we focus on one solid torus $V=V_i$, and
drop the subscripts for convenience. We describe the metric on $V$
that was constructed in \cite[Theorem 2.1]{fkp:jdg}. Let $\tlim =
1-(2\pi/\ell(s_i))^2$, and choose a parameter $t \in (0, \tlim$).
For fixed $t$, the metric on $V$ can be described in cylindrical
coordinates by the equation
\begin{equation}\label{eq:fkp-metric}
ds^2 = dr^2 + (f_{t}(r))^2 \, d\mu^2 + (g_{t}(r))^2 \, d\lambda^2,
\end{equation}
Here, $f_t$ and $g_t$ are functions obtained by solving a certain ODE.
The coordinate $0\leq \mu \leq 1$ is measured around each meridional
circle, while $0 \leq \lambda \leq 1$ is measured perpendicular to
$\mu$ in the longitudinal direction. Furthermore, $r_0(t) \leq r \leq 0$ is
radial distance, where $r = 0$ on the boundary torus $\partial V$ and $r_0(t) < 0$ is the unique root of the function $f_{t}(r)$. The radial value $r = r_0(t)$ 
corresponds to the core of the solid torus $V$.

Observe that the expression for the metric in equation \eqref{eq:fkp-metric}
is already diagonalized, with the three coordinate vectors mutually orthogonal. 
Furthermore, if we lift the metric to the universal cover $\widetilde{V}$, the $\mu$ and $\lambda$ coordinates are globally defined. Thus the projection from $\widetilde{V} \cong D \times \RR$ to the cross-sectional disk $D$, defined by $(r , \mu, \lambda) \mapsto (r, \mu)$, is a Riemannian submersion, as claimed in \eqref{item:sym}. 

To prove the theorem, it remains to compute the area of a meridian
disk $D$.  Since the coordinate $\lambda$ is constant over a meridian
disk, the area for fixed $t$ will be given by 
$$\int_0^1\int_{r_0(t)}^0 f_t(r)\, dr d\mu = \int_{r_0(t)}^0 f_t(r)\, dr.$$

By \cite[Theorem 5.4 and equation (4)]{fkp:jdg}, we know that as $t
\to \tlim$, the functions $f_{t}$ converge uniformly to
$$f_{\tlim}(r) = \frac{\ell(s_i)\sqrt{1-\tlim}}{\sqrt{\tlim}}
\sinh(\sqrt{\tlim}(r-r_0)), \quad \mbox{where} \quad r_0 =
-\tanh^{-1}(\sqrt{\tlim})/\sqrt{\tlim}. $$
Thus, as $t \to \tlim$, the area of the meridian disk limits to 
{\setlength{\jot}{1.3ex}
\begin{eqnarray*}
  \lim_{t\to \tlim} \int_{r_0(t)}^0 f_{t}(r) \, dr &
  = & \int_{r_0}^0 f_{\tlim}(r) \, dr \\
   & = &  \int_{r_0}^0 \frac{\ell(s_i)\sqrt{1-\tlim}}{\sqrt{\tlim}}
\sinh \left( \sqrt{\tlim}(r-r_0) \right) dr \\
& = & \frac{\ell(s_i)\sqrt{1-\tlim}}{\tlim}\left(
\cosh(\sqrt{\tlim}(-r_0)) - 1\right) \\
& = & \frac{\ell(s_i)\sqrt{1-\tlim}}{\tlim}\left(
\cosh(\tanh^{-1}(\sqrt{\tlim})) - 1\right) \\
& = & \frac{\ell(s_i)\sqrt{1-\tlim}}{\tlim}\left(
\frac{1}{\sqrt{1-\tlim}} - 1\right) \\
& = & \frac{\ell(s_i) \cdot 2\pi/\ell(s_i)}{1-(2\pi/\ell(s_i))^2} \left( \frac{\ell(s_i)}{2\pi} - 1 \right) \\
&=& \frac{\ell(s_i)^2}{ \ell(s_i) + 2\pi} .
\end{eqnarray*}}

Therefore, for $t$ sufficiently close to $\tlim$, we obtain item
\eqref{item:merid}.  
\end{proof}

\section{Sweepouts and generalized Heegaard splittings}\label{sec:sweepouts}

In this section, we review the definitions of strongly irreducible
Heegaard splittings, and untelescoping for weakly reducible Heegaard
splittings.  We recall the definition of a sweepout, as well as
bounded area sweepouts.  For strongly irreducible Heegaard splittings,
the existence of bounded area sweepouts follows from results announced by Pitts and Rubinstein
\cite{pitts-rubinstein:min2}.  When the Heegaard splitting is not strongly
irreducible, we still obtain a bounded area sweepout after
untelescoping.

None of the results in this section are original. However, since the ideas described here are gathered from many sources, we found it helpful to write down a unified exposition.

\begin{define}\label{def:compression}
A \emph{compression body} $C$ is constructed by taking a closed,
oriented (possibly disconnected) surface $S$ with no $S^2$ components,
thickening it to $S \times [0,1]$, and attaching a finite number of
$1$--handles to $S\times \{1\}$ in a way that makes the result
connected.  The \emph{negative boundary} of $C$ is $\bdy_- C:= S
\times \{0\}$, and the \emph{positive boundary} is $\bdy_+ C := \bdy C
\setminus \bdy_- C$.

A handlebody of genus $g$, constructed by attaching $g$ $1$--handles to a $3$--ball, is also considered a compression body. Its negative boundary is empty, and its positive boundary is the genus $g$ surface $\bdy C$.

The \emph{spine} of a compression body $C$ consists of the negative
boundary $\bdy_- C$, along with the core arcs of the attached
$1$--handles. In the special case where $C$ is a handlebody, a spine
is any graph whose regular neighborhood is $C$.
In either case, a compression body $C$ deformation retracts to its
spine, and in fact, the complement of the spine in $C$ is homeomorphic to $\bdy_+ C
\times (0,1)$.
\end{define}

\begin{define}\label{def:heegaard}
A \emph{Heegaard splitting} of a compact orientable 3--manifold $M$ is an expression $M = C_1 \cup C_2$, where $C_1$ and $C_2$ are compression bodies glued along their positive boundaries. The surface
$$\Sigma = C_1 \cap C_2 = \bdy_+ C_1 = \bdy_+ C_2.$$
is called the  \emph{Heegaard surface} of the splitting $M = C_1 \cup C_2$. We often speak of the Heegaard splitting and its surface interchangeably, as one determines the other.

A Heegaard splitting $M = C_1 \cup C_2$
defines a \emph{sweepout}. This is a map $f\co M \to [-1,1]$,
such that:
\begin{itemize}
\item $f^{-1}(-1)$ is a spine of $C_1$ and $f^{-1}(1)$ is a spine of $C_2$.
\item For each $t \in (-1,1)$, $f^{-1}(t)$ is a surface isotopic to $\Sigma = \bdy_+ C_1 = \bdy_+ C_2$.
\end{itemize}
Writing $\Sigma_t = f^{-1}(t)$, we obtain a $1$--parameter family of surfaces isotopic to $\Sigma$, interpolating between the two spines.
\end{define}

\begin{define}\label{def:weak-strong}
A Heegaard splitting $M = C_1 \cup C_2$ is called \emph{reducible} if
there are properly embedded
disks $D_i \subset C_i$ that are essential (i.e.\ $\bdy D_i \subset
\bdy_+ C_i$ is nontrivial on $\bdy_+ C_i$) and such that $\bdy D_1 =
\bdy D_2 \subset \Sigma$.  If $M$ is irreducible, the sphere $D_1 \cup
D_2$ must bound a $3$--ball, hence $\Sigma$ is obtained by
 adding extra handles inside the
$3$--ball, a process called \emph{stabilization}.  Otherwise, if the splitting is not reducible, it is
called \emph{irreducible}.

Following Casson and Gordon
\cite{cg:heegaard}, we say a Heegaard splitting is \emph{weakly
  reducible} if there exist essential compression disks $D_i \subset C_i$ such that $\bdy D_1 \cap  \bdy D_2 = \emptyset$.   It is
\emph{strongly irreducible} otherwise.
\end{define}

Casson and Gordon showed that an irreducible manifold $M$ with a
non-stabilized weakly reducible splitting must contain an
incompressible surface \cite{cg:heegaard}.  Scharlemann and Thompson
showed that in this setting, we may always cut $M$ along
incompressible surfaces and obtain a strongly irreducible
\emph{generalized Heegaard splitting} \cite{scharlemann-thompson}. Our
description of these splittings follows Lackenby's paraphrase
\cite{lackenby:tunnel-alg}.

\begin{define}\label{def:gen-heegaard}
A \emph{generalized Heegaard splitting} of a compact orientable
$3$--manifold $M$ is a decomposition of $M$ into submanifolds
$W_1, \ldots, W_m$, for an even $m$, where each $W_i$ is a disjoint
union of compression bodies, glued along their boundaries as
follows. For each $0 < i < m/2$, we have
$$\bdy_- W_{2i} \cap M^\circ = \bdy_- W_{2i+1} \cap M^\circ, \qquad
\mbox{and} \qquad \bdy_+ W_{2i} = \bdy_+ W_{2i-1}.$$
(Here, the notation $M^\circ$ indicates the interior of $M$.) 
The surface $F_i = W_i \cap W_{i+1}$ is called an \emph{even} or \emph{odd} surface, depending on the parity of $i$. 
By construction, each component of the odd surface $F_{2i-1}$ is a Heegaard surface of the corresponding component of $(W_{2i-1} \cup W_{2i})$. 
\end{define}

Using the work of Casson and Gordon \cite{cg:heegaard}, Scharlemann
and Thompson proved that one may start with a $3$--manifold $M$ and an
irreducible splitting $\Sigma$, and construct a generalized Heegaard
splitting with the following properties:
\begin{enumerate}
\item \label{item:odd} Every component of every odd surface $F_{2i-1}$
  is a strongly irreducible Heegaard surface for the component of
  $(W_{2i-1} \cup W_{2i})$ that contains it. 
\item \label{item:even} Every even surface $F_{2i}$ is incompressible, with no 2--sphere
  components.
\item \label{item:euler} Every $F_i$ satisfies $\chi(\Sigma) \leq \chi(F_i) \leq 0$.
\end{enumerate}
A generalized Heegaard splitting with these properties is called \emph{thin}.

The process of constructing a generalized Heegaard splitting is called \emph{untelescoping}, and the inverse process (which recovers $\Sigma$) is called \emph{amalgamation}. We describe the amalgamation process briefly, since we will need it below.

Choose an even surface in the generalized splitting, say $F_2$.  By
Definition \ref{def:gen-heegaard}, a component $S_2 \subset F_2$ is
the negative boundary of a compression body $C_2 \subset W_2$ and a
compression body $C_3 \subset W_3$. By Definition
\ref{def:compression}, $C_2$ is constructed by attaching $1$--handles
to a surface $S \times \{1\}$, which is parallel to $S_2 = S \times \{
0 \}$. Extend these $1$--handles through $S \times [0,1]$, and attach
them directly to $S_2$. Similarly, extend the $1$--handles of the
compression body $C_3 \subset W_3$, to attach them directly to $S_2
\subset F_2$. This can be done while keeping the attaching disks
disjoint from the attaching disks of the $1$--handles in $C_2$. See
\cite[Figure 12]{lackenby:tunnel-alg}.

If we perform this construction for every component of $F_2$, the resulting surface $F_2'$,  obtained from $F_2$ by attaching annuli, is now a Heegaard surface of $ W_1 \cup \ldots \cup W_4$. In other words, we have obtained a generalized Heegaard splitting with fewer pieces. Continuing in this manner, we amalgamate all the pieces and obtain a Heegaard surface for $M$.
Although there are various choices involved this procedure (choosing handle structures on the compression bodies, choosing an order in which to amalgamate), amalgamation is guaranteed to recover the original Heegaard surface $\Sigma \subset M$. See \cite[Proposition 3.1]{lackenby:tunnel-alg}.

Our goal is to work with generalized Heegaard splittings in the context of a negatively curved metric, as constructed in Section \ref{sec:Dehn}. 

\begin{define}\label{def:almost-minimal}
Let $M$ be an orientable Riemannian $3$--manifold. An orientable surface $F \subset M$ is called \emph{almost minimal} if $F$ is either a minimal surface, or is the boundary of an $\varepsilon$--neighborhood of a non-orientable minimal surface $F'$.
\end{define}

We will use the following result on essential surfaces and minimal
surfaces; see \cite{freedman-hass-scott, meeks-simon-yau, schoen-yau}.

\begin{lemma}\label{lemma:even-minimal}
Let $M$ be a negatively curved Riemannian $3$--manifold, whose
boundary (if any) consists of minimal surfaces.  Then
\begin{enumerate}
\item A connected, orientable, essential surface $F \subset M$ is
  isotopic to an almost minimal surface, as in Definition
  \ref{def:almost-minimal}.
\item If $F$ and $G$ are disjoint, connected, non-parallel essential surfaces,
  then their almost--minimal representatives are also disjoint.
\end{enumerate} 
\end{lemma}

In particular, if $\{ W_1, \ldots, W_m \}$ form a thin generalized
Heegaard splitting for $M$, one may isotope the $W_j$ into a position
where each piece $M_i = W_{2i-1} \cup W_{2i}$ has minimal boundary. A
boundary surface of $M_i$ may be part of the original boundary of $M$,
or it may result from cutting along a minimal surface corresponding to
an even $F_i$, as in Lemma \ref{lemma:even-minimal}.
(Note that if $F_i$ is isotopic to the orientable double cover of a
non-orientable minimal surface $F'_i$, then cutting along $F'_i$ is
equivalent to cutting along $F_i$, and produces orientable, minimal
boundary.)

Every minimal surface in $M$ satisfies an upper bound on area, related to its curvature.

\begin{lemma}\label{lemma:minimal-surf-area}
Let $M$ be an orientable Riemannian $3$--manifold whose  sectional curvatures are bounded above by $\kmax < 0$. Let $S \subset M$ be a minimal surface. Then
\begin{equation}\label{eq:min-area}
\area(S) \: \leq \: 2\pi \chi(S) / \kmax.
\end{equation}
\end{lemma}

\begin{proof}
This follows from the Gauss--Bonnet theorem, combined with the properties of minimal surfaces.  Let  $\lambda_1, \lambda_2$ denote the
principal normal curvatures at a point $x \in S$, let $K_x$ denote the sectional curvature of $M$ along $S$, and let $\kappa$ denote the Gaussian curvature of $S$ at $x \in S$.
Then the minimality of $S$ implies  that
$\lambda_1 = -\lambda_2$
hence 
$$\kappa \: = \: K_x + \lambda_1\lambda_2  \: \leq \: K_x  \: \leq \: \kmax.$$
Integrating over $S$, we obtain
$$\kmax \, \area(S) \: = \: \int_S \kmax \, dA \: \geq \:  \int_S \kappa \, dA 
\: = \: 2 \pi \, \chi(S),$$
where the last equality is the Gauss--Bonnet theorem. Dividing the
previous equation by $\kmax$ gives the desired statement.
\end{proof}

Note that if $F$ is the orientable double cover of a non-orientable
minimal surface $F'$, as in Definition \ref{def:almost-minimal}, then
choosing a sufficiently small $\varepsilon$--neighborhood ensures that
$\area(F)$ satisfies a bound arbitrarily close to \eqref{eq:min-area}.
This will be useful for Heegaard surfaces.

The following lemma is essentially a reformulation of an announced
result by Pitts and Rubinstein \cite{pitts-rubinstein:min2}.  As
mentioned in the introduction, a complete proof of their result has
not yet appeared.

\begin{lemma}\label{lemma:sweepout}
Let $M$ be an orientable Riemannian $3$--manifold whose non-compact
ends (if any) are isometric to horospherical cusp neighborhoods, and
whose boundary (if any) consists of minimal surfaces.  Suppose the
sectional curvatures of $M$ are bounded above by $\kmax < 0$.  Let
$\Sigma$ be a strongly irreducible Heegaard surface for $M$. Then, for
any $\zeta \in (0,1)$, there is a sweepout of $M$ corresponding to
$\Sigma$, such that every level surface $\Sigma_t$ satisfies
$$\area(\Sigma_t) \: \leq \: \: \frac{ 2\pi \chi(\Sigma) } {\zeta \, \kmax }.$$
\end{lemma}
  
\begin{proof}
As described in Definition \ref{def:heegaard}, the Heegaard surface
$\Sigma$ specifies a sweepout of $M$.  In any choice of sweepout,
there is a surface $\Sigma_t$ of maximal area.  Let $A$ be the infimum
of these maximal areas, over all sweepouts corresponding to
$\Sigma$. The number $A$ is called a \emph{minimax value}.

The work of Pitts and Rubinstein implies that there is a minimal surface $F
\subset N$ whose area is the minimax value $A$ if $F$ is orientable,
or $A/2$ if $F$ is non-orientable \cite{pitts-rubinstein:min2}.
Furthermore, $\Sigma$ is constructed by taking an almost minimal
surface corresponding to $F$, and then possibly attaching a single
tube.  Both of these operations (taking the boundary of an
$\varepsilon$--neighborhood, attaching a tube) can be achieved while
increasing the area $A$ by an arbitrarily small amount. Thus, Lemma
\ref{lemma:minimal-surf-area} guarantees that for any $\zeta \in
(0,1)$, there is a sweepout whose level surfaces satisfy
$$\area(\Sigma_t) \: \leq \: \:  2\pi \chi(\Sigma) / \zeta \, \kmax ,$$
as desired.
\end{proof}

\section{Core of Dehn filling and Heegaard surfaces}\label{sec:breslin}

The goal of this section is to prove Theorem \ref{thm:intro-breslin}, which was stated in the introduction. Most of the work here goes into proving the following, slightly simpler statement.

\begin{theorem}\label{thm:core-gen-splitting}
Let $M$ be an orientable Riemannian $3$--manifold whose non-compact
ends (if any) are isometric to horospherical cusp neighborhoods, and
whose boundary (if any) consists of minimal surfaces.  Suppose the
sectional curvatures of $M$ are bounded above by $\kappa_{\max} < 0$.
Let $\{W_1, \ldots, W_m\}$ be a thin generalized Heegaard splitting of
$M$, with separating surfaces $F_i = W_i \cap W_{i+1}$.

Let $V$ be a solid torus in $M$, such that the metric on $V$ is submersible. 
If a cross--sectional disk $D$ of $V$ satisfies
\begin{equation}\label{eq:core-gen-splitting}
\area(D) \: > \: \frac{2\pi \, \chi(F_i)} {\kappa_{\max}} \quad \forall i,
\end{equation}
then the core curve $\gamma$ of $V$ is isotopic into one of the odd
surfaces $F_i$.
\end{theorem}

The statement of Theorem \ref{thm:core-gen-splitting} is similar to
Theorem \ref{thm:intro-breslin}, with the single difference that a
Heegaard surface $\Sigma$ has been replaced by a thin generalized
Heegaard splitting. In our applications, this will be the generalized
splitting obtained by untelescoping $\Sigma$.  Furthermore, by
amalgamating the generalized splitting to recover $\Sigma$, we will
see that Theorem \ref{thm:intro-breslin} follows quickly from this
statement.

For the remainder of the section, we will use the definitions and
notation of Theorem \ref{thm:core-gen-splitting}. In particular, $M$
will denote a negatively curved $3$--manifold, with a generalized
Heegaard splitting $\{W_1, \ldots, W_m\}$. As in the theorem, $V$ will
denote a submersible solid torus in $M$, whose cross-sectional disk
satisfies \eqref{eq:core-gen-splitting}.

The proof of Theorem \ref{thm:core-gen-splitting} is essentially due to
Breslin, and follows the same line of argument as in his paper
\cite{breslin}.  We need to make slight modifications to his argument
to accommodate manifolds with boundary, generalized Heegaard splittings, and the metric of variable negative curvature. However, the spirit of the argument is the same.  Where our argument differs from his, we walk through the details carefully.

The proof breaks down into the following claims:
\begin{enumerate}
\item The core curve $\gamma$ of $V$ is isotopic into the complement of the even surfaces. This means that we can work with a single odd surface $S = F_i$, which is strongly irreducible in its component. 
We will show this in Lemma \ref{lemma:VintF-trivial}.
\item The Heegaard surface $S = F_i$ contains a simple loop homotopic to $\gamma^n$ for some $n$. 
The geometric proof of this claim, based on \cite[Lemma 2]{breslin}, appears in Lemma \ref{lemma:breslin-lemma2}.
\item There is an embedded annulus $A \subset M$, such that one component of $\bdy A$ is on $S$ and the other component of $\bdy A$ is on $\bdy N(\gamma)$, a tubular neighborhood of $\gamma$. 
 The  argument is nearly the same as that of   \cite[Lemma 1]{breslin}, and is recalled in Lemma  \ref{lemma:breslin-lemma1}.  
\item The loop $\bdy A \cap \bdy N(\gamma)$ is isotopic to $\gamma$, hence $\gamma$ is isotopic into $S$ through $A$. The proof is identical to  \cite[Lemma 5]{breslin}, which we restate in Lemma  \ref{lemma:breslin-lemma5}.  
\end{enumerate}

The following lemma is useful for both incompressible surfaces and
strongly irreducible Heegaard surfaces.  The proof is inspired by
\cite[Claim 1]{breslin}.

\begin{lemma}\label{lemma:gamma-miss-sfce}
Let $M$ and $V$ be as in Theorem \ref{thm:core-gen-splitting}. Suppose
that $F$ is a compact, orientable surface embedded in $M$, such that
$$\area(F) \: < \: \area(D),$$
where $D$ is the cross-sectional disk in $V$. Then $V \setminus F$
contains a closed curve that is essential in $V$. 
\end{lemma}

\begin{proof}
If $F \cap V$ contains a closed curve that is essential in $V$, we may
homotope this curve to one side of $F$, into $V \setminus F$,
satisfying the conclusion of the lemma. Thus we may assume that $F
\cap V$ does not contain an essential curve.

Consider lifts of $F \cap V$ in $\widetilde{V}$, the universal cover
of $V$.  Since $F \cap V$ does not contain an essential curve, there
is a lift $\widetilde{F}$ of $F\cap V$ in $\widetilde{V}$ that is
isometric to $F \cap V$.  Because $\widetilde{V}$ is a ball, any
connected component of $\widetilde{F}$ must separate $\widetilde{V}$.
We want to show that the two ends of $\widetilde{V}$ are contained in
the same component of $\widetilde{V} \setminus \widetilde{F}$.

Let $D$ be a cross-sectional disk of $\widetilde{V}$. By Definition
\ref{def:submersible}, there is a Riemannian submersion $f\co
\widetilde{V} \to D$. Because Riemannian submersions reduce area, and
$\area(F \cap V) < \area(D)$, the projection of $\widetilde{F}$ must
miss some point $x \in D$. Then the fiber $\{x \} \times \RR \subset
\widetilde{V}$ is disjoint from $\widetilde{F}$, hence no component of
$\widetilde{F}$ can separate the ends of $\widetilde{V}$.

The lift of a meridian disk to $\widetilde{V}$ consists of a disjoint
family of disks $D_j$, for $j \in \mathbb{Z}$, where a deck
transformation maps $D_j$ to $D_{j+1}$. For a natural number $n$, let
$R = R_n$ be the component of $\widetilde{V} \setminus (D_{-n} \cup
D_n)$ that has compact closure.  Because $F$ is compact, there are
finitely many lifts $\widetilde{F}_1, \dots, \widetilde{F}_k$ of $F
\cap V$ that intersect the closure of $R$.  Each component
$\widetilde{F}_i$ splits $R$ into two pieces, one of which does not
separate the ends of $\widetilde{V}$.  Thus the set $\widetilde{V}
\setminus (\widetilde{F}_1 \cup \dots \cup \widetilde{F}_k)$ contains
a component that intersects both $D_{-n}$ and $D_{n}$. In other words,
there is a path from $D_{-n}$ to $D_{n}$ that is disjoint from the
complete preimage of $F \cap V$. Since $n$ was arbitrary, it follows
that some component of the preimage of $V\setminus F$ in
$\widetilde{V}$ has non-compact closure. The projection of this
component to $V$ must contain a non-trivial curve.
\end{proof}

Using Lemma \ref{lemma:gamma-miss-sfce}, we can show that the core
curve $\gamma$ of $V$ is disjoint from the even, incompressible
surfaces in the generalized Heegaard splitting.

\begin{lemma}\label{lemma:VintF-trivial}
Let $M$ and $V$ be as in Theorem \ref{thm:core-gen-splitting}. Let $F
= F_2 \cup F_4 \cup \ldots \cup F_{m-2}$ be the union of the even
surfaces in the generalized Heegaard splitting of $M$, and suppose
(following Lemma \ref{lemma:even-minimal}) that each even $F_{i}$ has
been isotoped to be almost minimal in $M$. Suppose that the
cross-sectional disk of $V$ satisfies
$$
\area(D) \: > \: 2\pi \, \chi(F_{i}) / \kappa_{\max} \quad \forall i.
$$
Then some component $V'$ of $V \setminus F$ is a solid torus, isotopic
in $M$ to $V$ itself. 
Furthermore, each component of $V \setminus V'$ is a trivial
$I$--bundle over a subsurface of $\bdy V$.
\end{lemma}

\begin{proof}
By Lemma \ref{lemma:even-minimal}, each even $F_i$ is isotopic to a
minimal surface, or is the boundary of an $\varepsilon$--neighborhood of a
non-orientable minimal surface. 
Then Lemma \ref{lemma:minimal-surf-area} implies that 
for sufficiently small  $\varepsilon$, we can ensure 
the area of $F_i$ is
bounded above by $ 2\pi \chi(F_i) / \kappa_{\max}+ \delta$, for arbitrarily small $\delta$. In particular,
we have
$$\area(F_i) < \area(D).$$

Consider the curves of intersection $F \cap \bdy V$. If one of these
curves is a meridian of $V$, it must bound a disk in some $F_i$,
because $F_i$ is incompressible. 
Pass to an innermost disk in $F_i$, among all disks whose boundary is
a meridian curve on $\bdy V$. This disk must be a meridian of $V$,
because $M$ is negatively curved, hence irreducible.
But then $V \setminus F$ cannot contain a curve that is essential in
$V$, contradicting Lemma \ref{lemma:gamma-miss-sfce}. Therefore, every
curve of $F \cap \bdy V$ is either trivial in $\bdy V$, or else
essential in $V$.

As a preliminary step for the proof, we isotope trivial curves of
$F\cap \bdy V$ off $V$, through balls, as follows.
If any curve of $F \cap \bdy V$ is trivial in $\bdy V$, at least one
such trivial curve must be innermost in $F$.  This curve must bound a
disk $D_0 \subset F$ and a disjoint disk $D_1 \subset \bdy V$. Because
$M$ is irreducible, we may isotope $D_0$ past $D_1$, to remove this
curve of intersection.  If any trivial curves remain, repeat this
procedure with another innermost curve.  After this sequence of
isotopies, any remaining curves of $F \cap \bdy V$ are essential in
$V$.

Now, consider a component $A_0$ of $F \cap V$ (if any). Since each
component of $F$ is essential in $M$, and every curve of $\bdy A_0$ is
essential in $V$, the component $A_0$ must itself be incompressible in
$V$.  Since $\pi_1(A_0) \hookrightarrow \pi_1(V) \cong \ZZ$ and $F$ is
2--sided, this means $A_0$ is an annulus. As all annuli in a solid
torus are boundary--parallel, $F$ consists entirely of
boundary--parallel annuli.  Each boundary--parallel annulus $A_0$ cuts
off a boundary--parallel solid torus in $V$.

After removing all of these boundary--parallel pieces (if any), we
find a component $V'$ of $V \setminus F$ that is isotopic to $V$
itself.
Now, undo the isotopies through balls that removed the trivial curves
of intersection of $F \cap \bdy V$ in the preliminary step of the
proof.  Each of these isotopies modifies $V'$ by pushing a disk on the
boundary into or out of $V'$.  In particular, these isotopies preserve
the property that the component $V'$ of $V \setminus F$ is isotopic to
$V$ itself.  They also preserve the property that each component of $V
\setminus V'$ is boundary-parallel.
\end{proof}

Recall the thin generalized splitting $\{W_1, \ldots, W_m\}$ of $M$,
with surfaces $F_i = W_i \cap W_{i+1}$.  By Lemma
\ref{lemma:VintF-trivial}, we know that there is a solid torus $V'
\subset V$, which is isotopic to $V$ itself, and is disjoint from
every even surface $F_{2i}$. This means that $V'$ is contained in some
submanifold $W_{2i-1} \cup W_{2i}$, in the complement of the even
surfaces.  For the rest of this section, we take $W$ to be the
connected component of $W_{2i-1} \cup W_{2i}$ containing $V'$.

The following lemma, and its proof, was inspired by \cite[Lemma
2]{breslin}.

\begin{lemma}\label{lemma:breslin-lemma2}
Let $M$ and $V$ be as in Theorem \ref{thm:core-gen-splitting}, and let
$V'$ be as in Lemma \ref{lemma:VintF-trivial}. Let $S \subset
F_{2i-1}$ be a strongly irreducible Heegaard surface for the
submanifold $W$ that contains $V'$.  Then, after an isotopy of $S$,
there is a simple closed curve $\delta \subset S \cap V'$ that is
essential in $V'$.
\end{lemma}

\begin{proof}
By Equation \eqref{eq:core-gen-splitting}, the cross--sectional disk
$D$ of $V$ satisfies the strict inequality $\area(D) > 2 \pi \chi (S)
/ \kappa_{\max}$. Thus, for some $\zeta \in (0,1)$ near $1$, we have
$$
\area(D) > 2\pi\chi(S)/(\zeta \kappa_{\max}).
$$
For this value of $\zeta$, Lemma \ref{lemma:sweepout} implies that
there exists a sweepout $f\co W \to [-1,1]$ corresponding to $S$, such
that every level surface $S_t$ satisfies
\begin{equation}\label{eq:st-area-bound}
\area(S_t) \: \leq \:  2\pi \chi(\Sigma) / (\zeta\kappa_{\max} ) \:
< \: \area(D). 
\end{equation}

Let $V' \subset V \cap W$ be the solid torus guaranteed by Lemma
\ref{lemma:VintF-trivial}, which is isotopic to $V$.

\begin{claim}\label{claim:component-each-t}
For every $t \in (-1, 1)$, some component of $V' \setminus S_t$
contains a closed curve that is essential in $V'$.
\end{claim}

\begin{proof}[Proof of claim]
By Equation \eqref{eq:st-area-bound} and Lemma
\ref{lemma:gamma-miss-sfce}, we know that some component of $V
\setminus S_t$ contains a closed curve that is essential in $V$. What
needs to be shown is that this curve can be taken to lie in $V'$.

Let $G_t = S_t \cap V'$.  Since $\bdy V'$ consists of sub-surfaces of
$\bdy V$ and sub-surfaces of $F = F_2 \cup F_4 \cup \ldots \cup
F_{m-2}$, and since $S_t$ is disjoint from $F$, we know that $\bdy G_t
\subset \bdy V$, i.e.\ $G_t$ is properly embedded in $V$.

Since $V$ contains an essential closed curve in the complement of
$S_t$, it also contains such a curve in the complement of
$G_t$. Furthermore, by Lemma \ref{lemma:VintF-trivial}, $V \setminus
V'$ consists of trivial $I$--bundles over subsurfaces of $\bdy V$.
Thus we may homotope the essential closed curve through the interior
of $V$, away from these boundary--parallel pieces and into $V'$.
\end{proof}

For each $t \in (-1,1)$, let $H_t = f(S \times [-1,t)) \subset W$ and
  $J_t = f(S \times (t,1]) \subset W$.  Note the closure of $H_t$ is
the compression body below $S_t$, and the closure of $J_t$ is the
compression body above $S_t$.  Define
\begin{eqnarray*}
E_H := \{ t \in (-1,1) &:& H_t \cap V^\circ \mbox{ contains a closed curve that is essential in $V'$} \}, \\
E_J := \{ t \in (-1,1) &:& J_t \cap V^\circ \mbox{ contains a closed curve that is essential in $V'$} \}. 
\end{eqnarray*}
By Claim \ref{claim:component-each-t}, we know that every $t \in (-1,1)$ is
contained in either $E_H$ or $E_J$.

Since each of $H_t \cap V^\circ$ and $J_t \cap V^\circ$ is open in
$W$, we know that both $E_H$ and $E_J$ are open sets.  Furthermore,
for $ t$ close to $-1$, $H_t$ is a small regular neighborhood of a
spine of $W_{2i-1}$, hence these values of $t$ must be contained in
$E_J$.  Similarly, for $ t$ close to $1$, $J_t$ is a small regular
neighborhood of a spine of $W_{2i}$, hence these values of $t$ must be
contained in $E_H$.  Since both $E_H$ and $E_J$ are open and
non-empty, and their union $(-1,1)$ is connected, it follows that $E_H
\cap E_J \neq \emptyset$.

Let $r \in E_H \cap E_J$. Then, by construction, $H_r \cap V^\circ$
contains an essential loop that we call $\alpha_H$, and $J_r \cap
V^\circ$ contains an essential loop that we call $\alpha_J$.  There
exist integers $n,m$ so that $(\alpha_H)^n$ is homotopic to
$(\alpha_J)^m$ in $V$, and thus there is an immersed annulus $A$ in
$V'$ with boundary components $(\alpha_H)^n$ and $(\alpha_J)^m$.  This
implies that some loop $A \cap S_r$ is non-trivial in $A$, and hence,
by taking a sub-loop, if necessary, there exists an embedded essential
loop in $A \cap S_r$, which must therefore be essential in $V'$.
Since the sweepout surface $S_r$ is isotopic to $S$, we are done.
\end{proof}

The rest of the proof follows from two lemmas, whose topological
proofs are due to Breslin.  We need only check that his proofs carry
through in our setting.  The first is \cite[Lemma 1]{breslin}.

\begin{lemma}\label{lemma:breslin-lemma1}
Let the submanifold $W$ and the Heegaard surface $S \subset F_{2i-1}$
be as in Lemma \ref{lemma:breslin-lemma2}. Let $\gamma$ be the core
curve of $V$, as in Theorem \ref{thm:core-gen-splitting}.
Then, possibly after isotoping $\gamma$,
there is a regular neighborhood $N(\gamma)$ and an embedded
annulus $A$ in $W \setminus N(\gamma)$ with $\partial A = \alpha \cup
\alpha'$, where $\alpha$ is a simple essential nonmeridional loop in
the boundary of $N(\gamma)$ and $\alpha' \subset S$.
\end{lemma}

\begin{proof}
By Lemma \ref{lemma:breslin-lemma2}, we may isotope $S$ so that $S
\cap V'$ contains a simple loop $\alpha$ that is essential in $V'$,
where $V' \subset V$ is the solid torus of Lemma
\ref{lemma:VintF-trivial}.
Recall the sweepout $f\co W \to [-1,1]$, which was constructed in the
proof of Lemma \ref{lemma:breslin-lemma2}.  This map has the property
that $f^{-1}(t)= S_t$ is isotopic to $S$ for each $t \in (-1,1)$,
while the sets $f^{-1}(\pm 1)$ consist of the spines of the two
compression bodies.

The next step is to use the Rubinstein--Scharlemann graphic.  For
closed manifolds, the graphic is defined in
\cite{rubinstein-scharlemann:graphic1}.  However, we need to use the
Rubinstein--Scharlemann graphic for manifolds with boundary, as in
\cite{rubinstein-scharlemann:graphic2}.  In particular, one
modification is that we only consider the part of the sweepout that
runs from $S_{-1+\epsilon}$ to $S_{(1-\epsilon)}$, for some
$\epsilon>0$ sufficiently small. This way, we avoid problems as the
sweepout meets the boundary of $W$.

Let $g\co V' \to [0,1]$ be a smooth function that gives a sweepout of
the solid torus.  That is, for all $s \in (0,1)$, $g^{-1}(s)$ is a
surface isotopic to $\partial V'$, while $g^{-1}(1) = \partial V'$,
and $g^{-1}(0)$ is a closed curve isotopic to $\gamma$, the core of $V$.  Isotope
$\gamma$ to be $g^{-1}(0)$.
We will consider the function $g_t = g |_{S_t\cap V'}$, and apply work
of Cerf \cite{cerf} to isotope $f$ and $g$ so that $g_t$ is Morse for
all but finitely many $t$, and near--Morse otherwise.  The
Rubinstein--Scharlemann graphic $G$ is the set of points $(t, s) \in
[-1,1] \times [0,1]$ such that $s$ is a critical value of $g_t$.
While care must be taken for applications of the graphic for manifolds
with boundary (compare \cite{rubinstein-scharlemann:graphic1} and
\cite{rubinstein-scharlemann:graphic2}), to apply Breslin's proof we
only need the result from Cerf theory that if $(t_1, s_1)$ and $(t_2,
s_2)$ are in the same component of
$([-1+\epsilon,1-\epsilon]\times[0,1])\setminus G$, then the surface
$S_{t_1}$ is isotopic to $S_{t_2}$ via an isotopy that takes the loops
in $g^{-1}_{t_1}(s_1)$ to the loops in $g_{t_2}^{-1}(s_2)$.

Using the graphic described above, we obtain \cite[Lemma
  4]{breslin}. According to that lemma, one of two conclusions must
hold for some $t \in (-1,1)$:
\begin{enumerate}
\item \label{item:essential} $S_t \cap \partial V'$ contains a loop
  that is essential and non-meridional on $\partial V'$, or
\item \label{item:non-ess} $S_t \cap V'$ does not contain an essential
  loop of $S_t$.
\end{enumerate}

The proof of \cite[Lemma 4]{breslin} works verbatim with the slight
modification to the sweepout that was mentioned above.  In addition to
the Rubinstein--Scharlemann graphic, it uses Scharlemann's No Nesting
Lemma \cite{scharlemann:no-nesting}.  Note that this lemma applies
equally well to 3--manifolds with and without boundary.

In case \eqref{item:essential} above, we are essentially done with the
proof of Lemma \ref{lemma:breslin-lemma1}, as follows. Let $N(\gamma) =
g^{-1}[0,1-\epsilon]$ be a solid torus slightly inside $V'$. 
Let
$\alpha' \subset S_t \cap g^{-1}(1)$ be the loop guaranteed by
conclusion \eqref{item:essential}, and let $\alpha$ be the projection
of $\alpha'$ to $\bdy N(\gamma) = g^{-1}(1-\epsilon)$. Then the
product annulus between $\alpha$ and $\alpha'$ satisfies the lemma.

In case \eqref{item:non-ess} above, there is a $t \in (-1,1)$ so that
$S_t \cap V'$ does not contain an essential loop of $S_t$, and thus
each loop in $S_t \cap \partial V'$ bounds a disk in $S_t$.  By Lemma
\ref{lemma:breslin-lemma2}, another sweepout surface $S_r$ has the
property that $S_r \cap V'$ contains a simple loop $\alpha_r$ that is
essential in $V'$.  Since $S_r$ and $S_t$ are isotopic, there is an
embedded annulus $A$ with one boundary component equal to $\alpha_r
\subset S_r$ and the other equal to some $\alpha_t \subset S_t$.

Because all loops of $S_t \cap \bdy V'$ bound disks on $S_t$, we may
isotope $A$ to avoid these disks, hence $\alpha_t$ is disjoint from
$S_t \cap \bdy V'$.  The loop $\alpha_t$ must therefore be contained
in the interior of $V'$ or disjoint from $V'$.  Because $S_t \cap V'$
contains no essential loop of $S_t$, and because $\alpha_t$ is
isotopic to the essential loop $\alpha_r$, we must have $\alpha_t$ in
$W \setminus V'$.  Isotope $A$ slightly, if necessary, so that
$\alpha_r$ is contained in the interior of $V'$.  Then $A$ is an
annulus embedded in $W$ with $\bdy A = \alpha_r \cup \alpha_t$ with
$\alpha_r \subset V'$ and $\alpha_t \subset (W\setminus V')$.  Thus
$A$ meets $\bdy V'$ in an essential loop $\alpha'_t$.  Use the
embedded annulus bounded by $\alpha_t$ and $\alpha'_t$ to isotope
$S_t$ so that $S_t \cap V'$ contains a loop that is essential and
nonmeridional in $\bdy V'$. This puts us back into case
\eqref{item:essential}, hence the proof is complete.
\end{proof}

We have seen that there is an embedded annulus in $W$ with one
boundary component on a neighborhood of $\gamma$ and one on our
Heegaard surface $S = F_{2i-1}$.  This annulus must wrap around some
power of $\gamma$.  The final step toward proving Theorem
\ref{thm:core-gen-splitting} is to show that in fact, we may take the
annulus to wrap exactly once around $\gamma$.

\begin{lemma}\label{lemma:breslin-lemma5}
Let $W$ be a 3--manifold with Heegaard surface $S$, and let $\gamma$
be a simple loop in $W$ with regular neighborhood $N(\gamma)$.  Let
$\alpha$ be an essential nonmeridional loop in $N(\gamma)$, and
$\alpha'$ a loop in $S$.  If there is an embedded annulus $A \subset
W$ disjoint from $N(\gamma)$, with boundary $\alpha \cup \alpha'$,
then $\gamma$ is isotopic into $S$.
\end{lemma}

\begin{proof}
This is \cite[Lemma 5]{breslin}, and Breslin's proof goes through
verbatim.  We note that his proof uses a thin position argument of
Schultens \cite{schultens}, modified slightly by Breslin.  The
argument holds for manifolds with or without boundary, and indeed
Schultens' original application concerned manifolds with boundary.
Hence we refer the reader to \cite[Lemma 5]{breslin} for the proof.
\end{proof}

The referee informs us that Lemma \ref{lemma:breslin-lemma5} can also be 
proved via a straightforward application of the Daisy Lemma \cite{jaco-rubinstein:layered}.

\begin{proof}[Proof of Theorem \ref{thm:core-gen-splitting}]
Recall that $S \subset F_{2i-1}$ is a component of one of the odd
surfaces in a generalized Heegaard splitting of $M$. Now, the theorem
follows by Lemmas \ref{lemma:breslin-lemma1} and
\ref{lemma:breslin-lemma5}.
\end{proof}

The above results also give a quick proof of Theorem
\ref{thm:intro-breslin}, which was stated in the introduction.

\begin{proof}[Proof of Theorem \ref{thm:intro-breslin}]
Let $\Sigma$ be a genus $g$ Heegaard surface of $M$. If $\Sigma$ is
irreducible, then, as described in Section \ref{sec:sweepouts}, we may
untelescope $\Sigma$ to a thin generalized Heegaard splitting $\{W_1,
\ldots, W_m \}$. By property \eqref{item:euler} of thinness, the
cross-sectional disk $D$ of the submersible solid torus $V$ satisfies
$$
\area(D) \: > \: \frac{2\pi \chi(\Sigma) }{ \kmax} \: \geq \:
\frac{2\pi \, \chi(F_i)} {\kappa_{\max}} \quad \forall i.
$$
Thus, by Theorem \ref{thm:core-gen-splitting}, the core curve
$\gamma$ of $V$ is isotopic into some odd surface $F_{2i-1}$.
 
Now, we assume $\gamma \subset F_{2i-1}$, and amalgamate the
generalized Heegaard splitting to recover $\Sigma$. At certain times
during the amalgamation process, we will need to attach handles to a
partially amalgamated surface containing $\gamma$. Each time we do
this, a small isotopy ensures that $\gamma$ is disjoint from the disks
along which we attach handles. Thus, after the amalgamation is
complete, we have $\gamma \subset \Sigma$.

Meanwhile, if $\Sigma$ is reducible, then we destabilize $\Sigma$ to
an irreducible Heegaard surface $\Sigma'$ of genus $h < g$, apply the
above argument to $\Sigma'$, and then stabilize back to genus
$g$. Since stabilizations are unique, this approach recovers the
desired result for $\Sigma$.
\end{proof}

\section{Assembling the pieces}\label{sec:assembly}

In this section, we complete the proof of Theorem \ref{thm:main}. 
The following lemma will permit us to apply the results of Section \ref{sec:breslin}.

\begin{lemma}\label{lemma:area-D}
Let $C_1, \ldots, C_k$ be disjoint cusp neighborhoods of a hyperbolic
$3$--manifold $X$, and let $s_i$ be a slope on cusp $C_i$. Suppose
that the shortest slope length is $\lmin > 2\pi(2g-1)$, for some
$g\geq 2$.  Then the Dehn filling $M = X(s_1, \ldots, s_k)$ admits a
negatively curved metric as in Theorem
\ref{thm:negcurv-metric}, with curvatures bounded by $\kmax < 0$, such
that the area of a cross--sectional disk of every Dehn filling solid
torus satisfies
$$\area(D) > \frac{2\pi (2g-2)}{|\kmax|}.$$
\end{lemma}

\begin{proof}
If $\lmin > 2\pi(2g-1)$, then Theorem \ref{thm:negcurv-metric}
implies that for every $\zeta \in (0,1)$, the Dehn filled manifold
$M$ admits a Riemannian metric with sectional curvatures bounded by
$\kappa_{\max} = \zeta ( 4\pi^2/\lmin^2 - 1)$, and with
cross--sectional disks of area at least $\zeta \lmin^2/ (\lmin +
2\pi)$, where 
$$\lmin > 2\pi(2g-1).$$
Since the above inequality is strict, we may find $\zeta$ near $1$
such that $\lmin > (2\pi/\zeta^2)( 2g -2 + \zeta^2).$ Select this
value of $\zeta$ for the application of Theorem
\ref{thm:negcurv-metric}.

Then for any cross sectional disk $D$,
\begin{eqnarray*}
  \area(D) & \geq & \frac{\zeta \lmin^2}{\lmin + 2\pi} \: = \:
  \frac{\zeta(\lmin - 2\pi)}{1 - 4\pi^2/\lmin^2} \: \geq \:  \frac{\zeta^2(
    \lmin - 2\pi)}{|\kmax|} \\
  & > & \frac{\zeta^2\left( \frac{2\pi}{\zeta^2}(2g -2 +
    \zeta^2)- 2\pi \right)}{|\kmax|} \: = \: \frac{2\pi(2g-2)}{|\kmax|}. 
\end{eqnarray*}

\vspace{-3ex}
\end{proof}

\begin{proof}[Proof of Theorem \ref{thm:main}]
Let $\Sigma$ be a Heegaard surface of genus $g$ for the Dehn filled
manifold $M=X(s_1, \ldots, s_k)$.  
By Theorem \ref{thm:negcurv-metric}, $M$
admits a metric with sectional curvatures bounded by $\kmax < 0$, which means $g \geq 2$. By Lemma \ref{lemma:area-D}, the area of a cross--sectional disk
of each Dehn filling solid torus satisfies $\area(D) >
2\pi\chi(\Sigma)/\kmax$. Note that the solid tori constructed in Theorem \ref{thm:negcurv-metric} are submersible, as desired. Thus, by Theorem \ref{thm:intro-breslin}, we conclude that each core $\gamma_i$ of the $i^{\mathrm th}$ solid torus is isotopic into $\Sigma$, as required for conclusion \eqref{i:core-in-surf}.

We will prove conclusion \eqref{i:surf-from-unfilled} by induction on
$k$.  That is, let $\Sigma$ be a Heegaard surface for $M = X(s_1,
\ldots, s_k)$.  In the following argument, we will show that the core
$\gamma_k$ of the $k^{\mathrm{th}}$ solid torus can be isotoped off
$\Sigma$ in such a way that $\Sigma$ becomes a Heegaard surface for $M
\setminus \gamma_k \cong X(s_1, \ldots, s_{k-1})$.  This argument
works for arbitrary $k$.  Hence by induction, $\Sigma$ becomes a
Heegaard surface for $X$.

In the following argument, we may also assume without loss of generality that $\Sigma$ is irreducible. Otherwise, as in the proof of Theorem \ref{thm:intro-breslin}, we simply destabilize $\Sigma$ to an irreducible surface $\Sigma'$, apply the argument to $\Sigma'$, and stabilize at the end to recover $\Sigma$.

\smallskip

With these preliminaries out of the way, we assume that $\Sigma$ is irreducible in $X(s_1, \ldots, s_k)$, and untelescope $\Sigma$ to a thin generalized Heegaard splitting $\{W_1, \ldots, W_m \}$. Let $V = V_k$ be the $k^{\mathrm th}$ Dehn filling solid torus, with the negatively curved metric of Theorem \ref{thm:negcurv-metric}. By Lemma \ref{lemma:area-D}, the cross-sectional disk $D$ of $V$ satisfies 
$$
    \area(D) \: > \: \frac{2\pi \chi(\Sigma) }{ \kmax} \: \geq \:  \frac{2\pi \, \chi(F_i)} {\kappa_{\max}},
$$
for each surface $F_i$ of the generalized splitting. Thus, by Theorem \ref{thm:core-gen-splitting}, the core curve $\gamma_k$ of $V = V_k$ is isotopic into some odd surface $ F_{2i-1}$. Recall that $S = F_{2i-1}$ is a strongly irreducible Heegaard surface of the submanifold $W = W_{2i-1} \cup W_{2i}$.

After $\gamma_k$ has been isotoped into $S$, consider the surface $S
\setminus N(\gamma_k) \subset W$, where $N(\gamma_k)$ is a small
tubular neighborhood of $\gamma_k$ contained in $W$.
Note $S \setminus V_k \subset W\setminus V_k$.  Let $\{D_1, \ldots,
D_n\}$ be a collection of disjoint, non-parallel compression disks for
$S \setminus N(\gamma_k)$ in $W$,
which is maximal with respect to
inclusion.  Since $S$ is strongly irreducible, all of the $D_j$ must
be contained in the same compression body, say $W_{2i-1}$. Let $S'
\subset W_{2i-1}$ be the surface obtained after compressing
$S \setminus N(\gamma_k)$
along all of the $D_j$.

\begin{claim}\label{claim:totally-compressible}
Each component of $S'$ is either a sphere, a closed surface parallel
to $\bdy_- W_{2i-1}$, or an annulus parallel to 
$N(\gamma_k)$.
\end{claim}

\begin{proof}[Proof of claim]
First, suppose that a component of $S'$ is closed and not a
$2$--sphere. Since we have compressed
$S \setminus N(\gamma_k)$
along a maximal collection of disks, this component must be
incompressible in $W_{2i-1}$. But the only incompressible surfaces in
a compression body are parallel to the negative boundary $\bdy_-
W_{2i-1}$.

Next, suppose that a component of $S'$ is a surface with boundary, and
call this component $R$. Since
$\bdy (S \setminus N(\gamma_k))$
consists of two curves, both parallel to $\gamma_k$, $\bdy R$ must be
a union of one or two curves parallel to $\gamma_k$. Since the surface
$S$ has genus at most $g$, and we obtained $R$ by
cutting $S$ along $\gamma_k$ and then compressing, the genus of $R$ is at most
$g-1$.

Suppose, for a contradiction, that $R$ is an essential surface in $W
\setminus N(\gamma_k)$: that is, not an annulus parallel to
$N(\gamma_k)$. Then, since $\bdy W = F_{2i-2} \cup F_{2i}$ is a union
of incompressible surfaces, the component $R$ must also be essential
in $X(s_1, \ldots, s_{k-1}).$ Isotope the cores $\gamma_1, \ldots,
\gamma_{k-1}$ until they intersect $R$ minimally, and let $R^\circ = R
\setminus (\gamma_1 \cup \ldots \cup \gamma_{k-1})$.

We claim
that $R^\circ$ must be incompressible in the original cusped
hyperbolic manifold $X$.  To see this, suppose that an essential curve
$\alpha \subset R^\circ$ bounds a compression disk $D_0 \subset X$. By
passing to an innermost sub-disk if needed, we may assume that the interior of $D_0$
is disjoint from $R$. But $R$ is incompressible in $X(s_1, \ldots,
s_{k-1})$, hence $\alpha$ also bounds a disk $D_1 \subset R$. Since
$\alpha = \bdy D_1$ is an essential curve in $R^\circ$, the disk $D_1$
must be punctured two or more times by the $\gamma_i$. On the other
hand, $D_0$ is disjoint from all the cores $\gamma_i$, and has the
same boundary as $D_1$. Isotoping these cores through the ball
co-bounded by $D_0$ and $D_1$, past $D_1$, will reduce the
intersection number between $R$ and $\gamma_1 \cup \ldots \cup
\gamma_{k-1}$, contradicting the construction of $R^\circ$. Therefore,
$R^\circ$ is incompressible in $X$.

Recall that by construction, $R^\circ$ has a boundary component on
$N(\gamma_k)$.  Thus, if $R^\circ$ is a boundary-parallel annulus,
then so is $R$, a contradiction.  Thus $R^\circ$ is incompressible and
boundary--incompressible in $X$.

Now, remove the horospherical cusps $C_1, \ldots, C_k$ from both $X$
and $R^\circ$, and consider $\bdy R^\circ$, which is a union of $b$
closed curves on $\bdy C_1, \ldots, \bdy C_k$. The curves of $ \bdy
R$, which run parallel to $\gamma_k$, must be one or two longitudes of
the filling slope $s_k$.  Meanwhile, every other curve of $\bdy
R^\circ$ is a meridian of some $\gamma_j$ (for $1 \leq j \leq k-1$),
hence is in the isotopy class of the filling slope $s_j$.

By the hypotheses of Theorem \ref{thm:main}, the shortest longitude of
$s_k$ has length $\ell(\lambda_k) > 6(2g-3)$.  Each filling slope
$s_j$ for $1 \leq j \leq k-1$ also has length
$$\ell(s_j) \: > \: 2\pi(2g-1) \: > \: 6(2g-1)  \: > \: 6(2g-3).$$  

We conclude that the total length $\ell(\bdy R^\circ)$ of all the
curves of $\bdy R^\circ$ must satisfy
$$b \cdot 6(2g-3) \: < \: \ell(\bdy R^\circ) \: \leq \: 6 \abs{
  \chi(R^\circ) } \: \leq \: 6 \left( 2 (g-1) +b -2 \right) \: = \:
6(2g+b -4),$$ where the second inequality is a theorem of Agol
\cite[Theorem 5.1]{agol:bounds} and Lackenby \cite[Lemma
  3.3]{lackenby:word}. Comparing the first and last terms, we obtain
\begin{eqnarray*}
b \cdot (2g-3) & < & 2g + b - 4 \\
2gb - 2g  -4b + 4 &< & 0 \\
2(g - 2)(b-1) & < & 0 ,
\end{eqnarray*}
which is a contradiction since $g \geq 2$ and $b \geq 1$. This
contradiction proves the claim.
\end{proof}

Recall that we obtained $S'$ from $S$ by compressing along a maximal
collection of disks $\{D_1, \ldots, D_n\}$. By strong irreducibility
of $S$, all of these disks are contained in the same compression body
$W_{2i-1}$. Because all of these disks are disjoint from $\gamma_k$,
we may isotope $\gamma_k$ into $W_{2i-1}$ while staying disjoint from
$\{D_1, \ldots, D_n\}$.

We claim that $W_{2i-1} \setminus N(\gamma_k)$ is itself a compression
body. This can be seen by building the compression body ``downward''
from its positive boundary $S = \bdy_+ W_{2i-1}$. We thicken the
surface $S$ into $S \times [0,1] \subset W_{2i-1}$, and attach a
$2$--handle along each curve on $S \times \{ 0 \}$ corresponding to
$\bdy D_j$ for each $j$. After attaching the $2$--handles, the
resulting negative boundary is exactly the surface $S'$, with its two
boundary curves joined together. By Claim
\ref{claim:totally-compressible}, the surface obtained after attaching
$2$--handles consists of spheres, closed surfaces parallel to $\bdy_-
W_{2i-1}$, and a single torus isotopic to $\bdy N(\gamma_k)$. Thus,
after capping off each $2$--sphere with a $3$--ball, we obtain a
compression body $W'_{2i-1}$, satisfying
$$\bdy_+ W'_{2i-1} \: = \: \bdy_+ W_{2i-1} \: = \: S, \qquad \qquad \bdy_- W'_{2i-1} \: = \: \bdy_- W_{2i-1} \cup \bdy N(\gamma_k).$$

We have just shown that the core curve $\gamma_k$ may be isotoped off
the generalized splitting surface, into $W_{2i-1}$, in such a way that
the submanifolds
$$\{W_1, \ldots, W_{2i}, W'_{2i-1}, W_{2i+2}, \dots, W_m \}$$
form a thin generalized Heegaard splitting of $M \setminus \gamma_k =
X(s_1, \ldots, s_{k-1})$. After amalgamating this generalized
splitting, we obtain a Heegaard surface $\Sigma' \subset M \setminus
\gamma_k$.

Recall that, by \cite[Proposition 3.1]{lackenby:tunnel-alg},
amalgamation produces a unique Heegaard surface. Thus $\Sigma'$ is
isotopic in $M = X(s_1, \ldots, s_{k})$ to the surface $\Sigma$
obtained by amalgamating the splitting $\{W_1, \ldots, W_m \}$. In
other words, we have isotoped $\Sigma$ into $X(s_1, \ldots, s_{k-1})$,
in such a way that it is still a Heegaard surface. Repeating the above
argument for the core $\gamma_{k-2} \subset X(s_1, \ldots, s_{k-1})$,
and so on, completes the proof of Theorem \ref{thm:main}.
\end{proof}

\bibliographystyle{hamsplain}
\bibliography{biblio}

\end{document}